\documentclass[12pt,twoside]{article}

\usepackage[T1]{fontenc}
\usepackage[latin1]{inputenc}
\usepackage[english]{babel}
\usepackage[babel]{csquotes}

\usepackage{cite}

\usepackage{amssymb}
\usepackage{amsmath}
\usepackage{amsthm}
\usepackage{latexsym}
\usepackage{graphicx}
\usepackage{mathrsfs}
\usepackage{mathrsfs}

\theoremstyle{definition}
\newtheorem{rmk}{Remark}[section]

\theoremstyle{plain}
\newtheorem{lem}{Lemma}[section]
\newtheorem{thm}{Theorem}[section]
\newtheorem{prop}{Proposition}[section]

\def\En{\mathbb{N}}

\def\Ar{\mathbb{R}}

\def\Pi{\mathbb{P}}

\def\beq{\begin{equation}}
\def\eeq{\end{equation}}

\def\rarr{\rightarrow}

\def\impl{\Rightarrow}

\def\aand{\text{\quad and \quad}}
\def\for{\text{\quad for \quad}}
\def\as{\text{\quad as \quad}}
\def\l|{\left\|}
\def\r|{\right\|}

\def\L2{L^2(\Omega)}
\def\H1{H^1(\Omega)}
\def\nL2#1{\l|#1\r|_{\L2}}
\def\nH1#1{\l|#1\r|_{\H1}}
\def\rarrw{\rightharpoonup}

\def\gL2{L^2(\Gamma_1)}
\def\ngL2#1{\l|#1\r|_{\gL2}}

\begin{document}
\begin{center}
   {\huge\rm A doubly nonlinear evolution problem\\[0.3cm]
   related to a model for microwave heating\footnote{{\bf Acknowledgment.}\quad\rm The author is grateful to Pierluigi Colli for suggesting the problem and for his expert advice and fundamental support throughout this project.}}
\\[1.5cm]
 {\bf Luca Scarpa}\\[0.2cm]
 Dipartimento di Matematica "F. Casorati"\\
 Universit\`a di Pavia \\
 Via Ferrata 1, 27100 Pavia, Italy\\
 E-mail: {\tt luca.scarpa01@ateneopv.it}
                   \end{center}
       \vspace{1cm}
         
\begin{abstract}
This paper is concerned with the existence and uniqueness of the solution to a doubly nonlinear parabolic problem
  which arises directly from a circuit model of microwave heating.
  Beyond the relevance from a physical point of view, the problem is very interesting also in a mathematical approach: in fact, it consists of a
  nonlinear partial differential equation with a further nonlinearity in the boundary condition. Actually, we are going to prove a general result:
  the two nonlinearities are allowed to be maximal monotone operators and then an existence result will be shown for the resulting problem.
\\[.5cm]
\end{abstract}

\pagestyle{myheadings}
\newcommand\testopari{\sc Luca Scarpa}
\newcommand\testodispari{\sc A doubly nonlinear evolution problem for microwave heating}
\markboth{\testodispari}{\testopari}


\thispagestyle{empty}

\section{Introduction}
\setcounter{equation}{0}

In this paper, we deal with a problem that stems from a circuit model of microwave heating (see \cite{reim-jor-min-var} for details);
in particular, we aim at proving the existence of a solution for a nonlinear partial differential equation with appropriate initial and boundary conditions.
More specifically, we consider a model of a RLC circuit in which a thermistor has been inserted: this one has a cylindrical shape and takes into account the
temperature's effects. A system of three equations is obtained in~\cite{reim-jor-min-var}; this system involves
the voltage $V$ across the capacitor, resistor and inductance, the potential $\Phi$ associated to the electrostatic field in the thermistor and the temperature~$\vartheta$.
In order to prove the existence of a solution of the entire system, the first step is showing the existence of a solution for the equation for $\vartheta$, 
complemented by a special boundary condition and initial condition.
This is the aim of our work.

Firstofall, the thermistor is modelled by a cylinder $C_l\subseteq\Ar^3$:
\beq
  \label{therm}
  D:=B_R(0)\subseteq\Ar^2\,, \quad C_l:=D\times(0,l)\,, \quad l,R>0
\eeq
\beq
  \label{boundary}
  B:=\partial D\times(0,l)\,, \quad
  D_I:=D\times\{0\}\,, \quad D_F:=D\times\{l\}\,.
\eeq
The equation in $\vartheta$ which arises from the model is the following (see \cite{reim-jor-min-var})
\beq
  \label{eq_theta}
  c_0\vartheta_t-\Delta K(\vartheta)=g \quad\text{in}\quad C_l\times(0,T)\,,
\eeq
where $T>0$, $c_0>0$, $g\in L^2(C_l\times(0,T))$ and $K:\Ar\rarr\Ar$ is a Lipschitz-continuous increasing function such that $K(0)=0$.
Please note that if we write \eqref{eq_theta} using the variable $u=K(\vartheta)$, we obtain
\beq
  \label{eq_u}
  c_0\left(\gamma(u)\right)_t-\Delta u=g  \quad\text{in}\quad C_l\times(0,T)\,,
\eeq
where 
\beq
  \label{K_op}
  \gamma:=K^{-1}:\Ar\rarr\Ar
\eeq
is the inverse graph of $K$. In the application we are considering, $\gamma$
is still a Lipschitz-continuous increasing function with $\gamma(0)=0$.

Finally, the appropriate boundary conditions arising from the model are
\beq
  \label{bound_cond_u}
  -\frac{\partial u}{\partial n}=
  \begin{cases}
    \beta(u)-\beta(u_B) \quad&\text{in $B\times(0,T)$}\,,\\
    0 \quad&\text{in $D_I\cup D_F\times(0,T)$}\,,
  \end{cases}
\eeq
where $u_B\in\Ar$ is a datum and
\beq
  \label{bound_op}
  \beta:\Ar\rarr\Ar\,, \quad \beta(r)=h\gamma(r)+s\left|\gamma(r)\right|^3 \gamma(r)\,, \quad h,s>0\,.
\eeq
Note that $\beta$ is a continuous increasing function such that $\beta(0)=0$.

In this paper, we are interested to discuss the initial and boundary value problem
\begin{gather}
\label{problem_1}
    c_0\left(\gamma(u)\right)_t-\Delta u=g  \quad\text{in}\quad C_l\times(0,T)\\
\label{problem_2}
    -\frac{\partial u}{\partial n}= \beta(u)-\beta(u_B) \quad\text{in $B\times(0,T)$}\\
\label{problem_3}
    -\frac{\partial u}{\partial n}=0 \quad\text{in $D_I\cup D_F\times(0,T)$}\\
\label{problem_4}
    \left(\gamma(u)\right)(0)=v_0 \quad\text{in}\quad C_l\,.
\end{gather}
The system \eqref{problem_1}--\eqref{problem_4} is indeed very interesting from a physical
point of view. As a matter of fact, the variable $u$ is strictly correlated to the temperature $\vartheta$ (we recall that $u=K(\vartheta)$):
more specifically, the boundary condition \eqref{problem_3} tells that the flux across the top and the bottom of the cylindrical thermistor is null,
while \eqref{problem_2} specifies the relation between the flux across the lateral surface of the thermistor and the environment's temperature.

Moreover, beyond the physical view, the problem \eqref{problem_1}--\eqref{problem_4} is relevant in a mathematical perspective as well. In
fact, we are concerned with a partial differential system with a double nonlinearity: the first is contained in the main equation \eqref{problem_1}, 
within the derivative with respect to $t$, and the second appears in the boundary condition \eqref{problem_2}. In this work, we are going to prove
the existence of a solution of a
more general problem, in which \eqref{problem_1}--\eqref{problem_4} appears as a special case.
More precisely, we will consider an open set $\Omega\subseteq\Ar^n$ instead of the three-dimensional
 cylinder  $C_l$, and we will write $\partial\Omega=\Gamma_0\cup\Gamma_1$,
where $\Gamma_0\cap\Gamma_1=\emptyset$: $\Gamma_0$ will be the generalization of $D_I\cup D_F$, while $\Gamma_1$ replaces $B$.
Furthermore, we will deal with two maximal monotone operators $\gamma$ and $\beta$ instead of the single-valued
increasing functions \eqref{K_op} and \eqref{bound_op} respectively.

In the second section, we will illustrate the main results and the strategy of the proofs: relying on an important result by Di Benedetto - Showalter (see \cite{diben-show}),
we will obtain a first lemma, which proves the existence for a simplified problem, and then we will use this to prove the theorem ensuring 
the existence of a solution to \eqref{problem_1}--\eqref{problem_4}. Finally, we will present a continuous dependence and uniqueness result.

The third section is devoted to the proof of the lemma: it consists essentially in checking the hypotheses of the result contained in \cite{diben-show}
for our specific framework. In this sense, the most interesting point is the use of a fixed point argument in order to control the maximal monotonicity
of a particular operator.

The fourth section contains the proof of the existence of a solution to our problem: we will show some uniform estimates in order to pass to
the limit accurately and thus obtain the desired result.

The fifth section contains the proof of the continuous dependence and uniqueness result: this one is provided in a simplified setting,
in which a linearity assumption for $\gamma$ is in order.


\section{Main results}
\setcounter{equation}{0}

In this section, we will focus our attention on the problem \eqref{problem_1}--\eqref{problem_4}
and we will illustrate the two main results which prove the existence of a solution and the continuous dependence theorem.
As we have anticipated, we abandon the specific framework of the previous section and consider a more abstract setting: in this sense, the most important generalization consists
in dealing with two general maximal monotone operators on $\Ar$, instead of considering the given single-valued increasing functions \eqref{K_op} and 
\eqref{bound_op}. Furthermore,
we also relax the hypothesis on the cylindrical shape of domain \eqref{therm}.

More precisely, we now work in the following setting:
\beq
  \label{domain}
  \Omega\subseteq\Ar^n \quad \text{smooth bounded domain}\,, \quad
  \Gamma:=\partial\Omega\,, \quad \Gamma=\Gamma_0\cup\Gamma_1\,, \quad \Gamma_0\cap\Gamma_1=\emptyset\,;
\eeq
\beq
  \label{gamma}
  \gamma:\Ar\rarr2^{\Ar} \quad \text{maximal monotone}\,;
\eeq
\beq
  \label{beta}
  \beta:\Ar\rarr2^{\Ar} \quad \text{maximal monotone}\,, \quad 0\in D(\beta)\,.
\eeq
As we have anticipated, \eqref{domain} is the natural generalization of environment \eqref{therm}-\eqref{boundary}: $\Gamma_0$ 
here plays the role of $D_I\cup D_F$,
while $\gamma$ and $\beta$ are the extensions of \eqref{K_op} and \eqref{bound_op}, respectively. Hence, it will be natural to assume good properties on $\gamma$
(such as linear boundedness), while $\beta$ should be treated as generally as possible, at least at the beginning.
The problem we are dealing with is the following:
\beq
  \label{eq}
  \frac{\partial}{\partial t}v-\Delta u=g \quad \text{in } \Omega\times(0,T)
\eeq
\beq
  \label{eq_v}
  v\in\gamma(u) \quad \text{in } \Omega\times(0,T)
\eeq
\beq
  \label{eq_gamma}
  -\frac{\partial u}{\partial n}\in\beta(u)-h \quad \text{in } \Gamma_1\times(0,T)\,, \qquad
  \frac{\partial u}{\partial n}=0 \quad \text{in } \Gamma_0\times(0,T)
\eeq
\beq
  \label{init_v}
  v(0)=v_0 \quad\text{in}\quad\Omega\,,
\eeq
where $g\in L^2(0,T;L^2(\Omega))$ and $h\in L^2(0,T;L^2(\Gamma_1))$. As a remark, note that $\frac{\partial}{\partial n}$ 
indicates the outward normal derivative on $\Gamma$
and that we have used no particular notation for the traces on $\Gamma_1$ or $\Gamma_0$ (the contest is quite clear).

We would like to write a variational formulation of problem \eqref{eq}--\eqref{init_v}.
Note that for all $z\in H^1(\Omega)$, integrating by parts and taking into account condition  \eqref{eq_gamma}, we have
\[
    -\int_\Omega{\Delta u(t) z\,dx}=
        \int_\Omega{\nabla u(t)\cdot\nabla z\, dx}+\int_{\Gamma_1}{\left(\xi(t)-h(t)\right)z\,ds}\,,
\]
for some $\xi(t)\in\beta(u(t))$ a.e. on $\Gamma_1$. Hence,
testing equation \eqref{eq} by an arbitrary function $z\in H^1(\Omega)$ and integrating, we reach
the following variational formulation:
\beq
  \label{var_for}
    \begin{split}
      \left<\frac{\partial v}{\partial t}(t),z\right>&+\int_\Omega{\nabla u(t)\cdot\nabla z\, dx}\,+
      \int_{\Gamma_1}{\xi(t)z\,ds}=\int_\Omega{g(t)z\,dx}+\int_{\Gamma_1}{h(t)z\,ds} \\
      & \text{ for almost all } t\in(0,T) \aand  \forall z\in H^1(\Omega)\,,
    \end{split}
\eeq
where
\beq
\label{eq_v_var}
  v\in\gamma(u) \quad \text{a.e. in } \Omega\times(0,T)\,,
\eeq
\beq
  \label{eq_gamma1_var}
  \xi\in\beta(u) \quad \text{a.e. in } \Gamma_1\times(0,T)\,.
\eeq

In order to prove the existence of a solution for problem \eqref{eq}--\eqref{init_v}, the idea is to accurately use an abstract result by Di Benedetto - Showalter,
which we briefly remind (see \cite{diben-show} for further details). In particular, in this sense, we will be forced to require more regularity on $\gamma$; however, as we have anticipated,
our assumptions will be acceptable if we keep in mind our application.
\begin{thm}
  \label{ben_show}
  Let $W$ be a reflexive Banach space and $V$ a Hilbert space which is dense and embedded compactly in $W$. Denote the injection
  by $i:V\rarr W$ and the dual operator (restriction) by $i^*:W^*\rarr V^*$. Define $\mathscr{A}:=i^*\circ\partial\phi\circ i$. Assume the following:
  \begin{itemize}
    \item{[A1]} The real-valued function $\phi$ is proper, convex and lower semicontinous on $W$, continuous at some point of $V$,
      and $\partial\phi\circ i:V\rarr W^*$ is bounded.
    \item{[A2]} The operator $\partial\phi\circ i:L^2(0,T;V)\rarr L^2(0,T;W^*)$ is bounded.
    \item{[B1]} The operator $\mathscr{B}:V\rarr V^*$ is maximal monotone and bounded.
    \item{[B2]} The operator $\mathscr{B}:L^2(0,T;V)\rarr L^2(0,T;V^*)$ is bounded and coercive, i.e.
      \beq
        \label{coerc}
        \lim_{\substack{\l|u\r|_{L^2(0,T;V)}\rarr+\infty \\ [u,v]\in\mathscr{B}}}{\frac{\int_0^T{\left<v(t),u(t)\right>}\,dt}{\l|u\r|_{L^2(0,T;V)}}}=+\infty\,.
      \eeq
  \end{itemize}
  Then, for each $f\in L^2(0,T;V^*)$ and $v_0\in Rg(\mathscr{A})$, there exists a triple $u\in L^2(0,T;V)$, $v\in H^1(0,T;V^*)$, $w\in L^2(0,T;V^*)$
  such that
    \begin{gather}
    \label{bs1}
       \frac{d}{dt}v(t)+w(t)=f(t)\,, \quad v(t)\in\mathscr{A}(u(t))\,, \quad w(t)\in\mathscr{B}(u(t)) \quad \text{a.e. } t\in[0,T] \\
       v(0)=v_0\,.
    \end{gather}
\end{thm}

In our specific framework, we want to apply Theorem \ref{ben_show} to the spaces $V=H^1(\Omega)$ and $W=L^2(\Omega)$:
to be precise, in the development of the work we will make the identification $L^2(\Omega)\cong L^2(\Omega)'$.
In the notation of the theorem, the clear choice of $f$ is the following:
\beq
  \label{f}
  \left<f(t),z\right>=\int_\Omega{g(t)z\,dx}+\int_{\Gamma_1}{h(t)z\,ds}\,, \quad \forall z\in H^1(\Omega)\,, \quad  \forall t\in[0,T]\,.
\eeq
Furthermore, as far as $\mathscr{A}$ is concerned, it is natural to introduce
\begin{gather}
  \label{Abar}
  \bar{\mathscr{A}}(u)=\{v\in L^2(\Omega)\,: v(x)\in\gamma(u(x)) \quad \text{a.e. } x\in\Omega\}\,, \\
  D(\bar{\mathscr{A}})=\{u\in H^1(\Omega)\,: \exists v\in L^2(\Omega) \quad\text{ such that } v(x)\in\gamma(u(x)) \quad \text{a.e. } x\in\Omega\}\,;
\end{gather}
thus, our choice of the operator $\mathscr{A}$ is
\beq
  \label{A}
  \mathscr{A}:=i^*\circ\bar{\mathscr{A}}\,, \quad D(\mathscr{A})=D(\bar{\mathscr{A}})\,.
\eeq
However, while most of the assumptions of the theorem are satisfied by these particular choices of $V$, $W$, $f$ and $\mathscr{A}$,
on the other side, the intuitive way of considering $\mathscr{B}$, i.e.
\begin{gather}
  \label{B_first}
  \left<\mathscr{B}(u),z\right>=\int_\Omega{\nabla u\cdot\nabla z\, dx}+
      \int_{\Gamma_1}{\xi z\,ds}\,, \quad \xi\in\beta(u) \text{ a.e. on } \Gamma_1\,, \quad \forall z\in H^1(\Omega)\,, \\
  D(\mathscr{B})=\{u\in\H1\,: \exists\xi\in\gL2\quad\text{ such that } \xi\in\beta(u) \text{ a.e. on } \Gamma_1\}\,,
\end{gather}
gives us many problems: in particular, the coercivity condition \eqref{coerc} is not evident at this level. So, we modify
expression \eqref{B_first} by adding a "correction term", which we hope will give us the required coercivity: more precisely, we consider
the following regularization of~$\mathscr{B}$,
\beq
  \label{B_def}
  \left<\mathscr{B}_\lambda(u),z\right>=\lambda\int_\Omega{uz}+\int_\Omega{\nabla u\cdot\nabla z\, dx}+
      \int_{\Gamma_1}{\beta_\lambda(u)z\,ds}\,, \quad \forall z\in H^1(\Omega)\,,
\eeq
where $\lambda>0$ is fixed and $\beta_\lambda$ indicates the Yosida approximation of $\beta$ (please note that we are not using a specific
notation for the traces on $\Gamma_1$). The idea is to apply Theorem~\ref{ben_show} to $\mathscr{B_\lambda}$ instead of $\mathscr{B}$, for $\lambda>0$
fixed. As a consequence, the theorem itself will not directly give us a solution $u$ for 
problem \eqref{var_for}--\eqref{eq_gamma1_var}, but we will obtain a solution $u_\lambda$ satisfying the respective modified problem.
The second step will be to prove estimates on $u_\lambda$ independent of $\lambda$, and then passing to the limit as $\lambda\rarr0$ and
finding a solution of the original problem.

\begin{rmk}
\label{properties}
  Let us recall some properties of the Yosida approximation which will be useful in the following sections:
  for further details about these remarks, the reader can refer to \cite{barbu,brezis}.
  
  Let $(H,\l|\cdot\r|)$ be a Hilbert space and $A$ a maximal monotone operator on $H$: then, for all $\lambda>0$,
  the Yosida approximation $A_\lambda$ of $A$ is defined as
  \beq
    \label{yos}
    A_\lambda=\frac{I-J_\lambda}{\lambda}\,, \quad \text{where $J_\lambda=\left(I+\lambda A\right)^{-1}$ is the resolvent}\,,
  \eeq
  while for all $x\in D(A)$, we will write $A^0x$ for the minimum-norm element of $Ax$. With these notations, the following properties hold:
  \begin{itemize}
    \item $A_\lambda$ is a monotone and single-valued operator such that for all $x\in H$, $A_\lambda x\in AJ_\lambda x$;
    \item $A_\lambda$ is Lipschitz continuous (thus maximal monotone) with Lipschitz constant $\leq\frac{1}{\lambda}$;
    \item for all $\mu,\lambda>0$, $\left(A_\mu\right)_\lambda=A_{\mu+\lambda}$
    \item for all $x\in D(A)$, $\left|A_\lambda x\right|\leq\left|A^0x\right|$ and $\lambda\mapsto\left|A_\lambda x\right|$ is 
             non increasing as $\lambda\searrow0$;
    \item as $\lambda\searrow0$, $\{A_\lambda x\}$ is bounded if and only if $x\in D(A)$, and in this case $A_\lambda x\rarr A^0x$.
  \end{itemize}
  Furthermore, we will use the following result: if $\varphi$ is a convex, proper and lower semicontinuous function on $H$ such that $A=\partial\varphi$,
  then
  \beq
    \label{domini}
    D(A)\subseteq D(\varphi)\subseteq \overline{D(\varphi)}\subseteq\overline{D(A)}\,.
  \eeq
  Finally, if we define
  \beq
    \varphi_\lambda(x)=\min_{y\in H}\left\{\frac{1}{2\lambda}\left|y-x\right|^2+\varphi(y)\right\}\,, \quad x\in H\,, \quad\lambda>0\,,
  \eeq
  then $\varphi_\lambda$ is convex, Fréchet-differentiable with differential $A_\lambda$ and we have
  \beq
    \label{prop_yos}
    \varphi_\lambda(x)\nearrow\varphi(x) \quad\text{as } \lambda\searrow0 \quad\forall x\in H\,.
  \eeq
\end{rmk}

To summarize, the results we are going to present in this section could be briefly described as follow:
the first theorem tells us that Theorem \ref{ben_show} can be applied with the particular choices \eqref{f}, \eqref{A} and \eqref{B_def},
the second theorem gives us a solution for the original problem  \eqref{var_for}--\eqref{eq_gamma1_var}, and the third ensures that 
a continuos dependence and uniqueness result holds.

\begin{thm}
  \label{lemma}
  Let $\Omega$, $\gamma$, $\beta$ be as in \eqref{domain}--\eqref{beta}, and also suppose that
  \beq
     \label{00}
     D(\gamma)=\Ar\,, \quad \gamma(0)\ni0\,, \quad \beta(0)\ni0\,,
  \eeq
  \beq
    \label{lin_bound}
    \exists\,C_1, C_2>0 \text{ such that } \left| y\right|\leq C_1\left|x\right|+C_2 \quad \forall x\in\Ar\,, \quad \forall y\in\gamma(x)\,.
  \eeq
  For $g\in L^2(0,T;L^2(\Omega))$ and $h\in L^2(0,T;L^2(\Gamma_1))$,
  let $f$, $\mathscr{A}$ and $\mathscr{B}_\lambda$ as in \eqref{f}, \eqref{A} and \eqref{B_def}. Then, for any given pair
  \beq
    \label{init_pair}
    (u_0,v_0)\in\left(\L2\right)^2 \quad \text{such that}\quad\text{$v_0\in\gamma(u_0)$ a.e. in $\Omega$}\,,
  \eeq
   there exist $u_\lambda\in L^2(0,T;H^1(\Omega))$
  and $v_\lambda\in H^1(0,T;H^1(\Omega)')\cap L^2(0,T;\L2)$ such that
  \beq
    \label{lem1}
    \frac{\partial}{\partial t}v_\lambda(t)+\mathscr{B}_\lambda (u_\lambda(t))=f(t) \quad\text{for a.e. } t\in[0,T]\,,
  \eeq
  \beq
    \label{lem2}
    v_\lambda(t)\in\mathscr{A}(u_\lambda(t)) \quad \text{for a.e. t }\in [0,T]\,,
  \eeq
  \beq
    \label{lem3}
   v_\lambda(0)=v_0\,.
  \eeq
\end{thm}

\begin{thm}
  \label{theorem}
  Let $\Omega$, $\gamma$, $\beta$, $f$ and $\mathscr{A}$ as in
  \eqref{domain}--\eqref{beta}, \eqref{f} and \eqref{A} respectively, and suppose
  that conditions \eqref{00} and \eqref{lin_bound} of the previous theorem hold. Furthermore, assume that
\begin{gather}
    \label{bilip}
    \gamma \text{ is a bi-Lipschitz continuous function}\nonumber \\
    \text{(i.e., both  $\gamma$ and $\gamma^{-1}$ are Lipschitz continuous)},
\end{gather}
\beq  
\label{ip_beta}
    D(\beta)=\Ar \aand \exists D_1, D_2>0 \text{ such that } \left|\beta^0(r)\right|\leq D_1\hat{\beta}(r)+D_2\,,
\eeq
\beq
    \label{ip_g}
    g\in L^2(0,T;\L2)\cap L^1(0,T;L^\infty(\Omega))\,, \quad h\in L^2(0,T;L^2(\Gamma_1))\,, 
\eeq
  where $\hat{\beta}$ is a proper, convex and lower semicontinuous function such that $\partial\hat{\beta}=\beta$, and
  $\beta^0(r)$ is the minimum-norm element in $\beta(r)$. Then, for any given pair
  \beq
    \label{init_pair2}
    (u_0,v_0)\in\left(\L2\right)^2 \quad \text{such that}\quad\text{$v_0=\gamma(u_0)$ a.e. in $\Omega$}
    \quad\text{and}\quad \hat{\beta}(u_0)\in L^1(\Omega)\,,
  \eeq
 there are $u\in L^2(0,T;H^1(\Omega))$,
  $v\in H^1(0,T;H^1(\Omega)')\cap L^2(0,T;\L2)$, $\xi\in L^2(0,T;\gL2)$ such that
  \beq
    \label{thm1}
    \begin{split}
      \left<\frac{\partial v}{\partial t}(t),z\right>&+\int_\Omega{\nabla u(t)\cdot\nabla z\, dx}+
      \int_{\Gamma_1}{\xi(t)z\,ds}=\int_\Omega{g(t)z\,dx}+\int_{\Gamma_1}{h(t)z\,ds} \\
      & \text{ for a.e. } t\in(0,T) \aand  \forall z\in H^1(\Omega)\,,
    \end{split}
  \eeq
  \beq
    \label{thm2}
    v(t)\in\mathscr{A}(u(t)) \quad \text{for a.e. t }\in(0,T)\,, \quad \xi\in\beta(u) \quad \text{a.e. in } (0,T)\times\Gamma_1\,,
  \eeq
  \beq
    \label{thm3}
    v(0)=v_0\,.
  \eeq
\end{thm}

\begin{rmk}
  We can notice that hypothesis \eqref{ip_beta} is not too restrictive on $\beta$ itself: as a matter of fact,
  $\beta$ is allowed to have polynomial growth or even a first order exponential growth.
\end{rmk}
\begin{rmk}
  \label{rem}
  Thanks to the assumption \eqref{bilip} on $\gamma$, the first inclusion in \eqref{thm2} yields  $v(t)=\gamma(u(t))$ for a.e. $t\in(0,T)$ (cf.~\eqref{Abar}--\eqref{A}), so that
  it is possible to deduce a higher regularity of $v$, that is
  \beq
    \label{reg1_v}
    v\in L^2(0,T;\H1)\,.
  \eeq
  Furthermore, since $v\in H^1(0,T;H^1(\Omega)')\cap L^2(0,T;\H1)$, we also have that
  \beq
    \label{reg2_v}
    v\in C^0([0,T];\L2)\,.
  \eeq
\end{rmk}

\begin{thm}
  \label{dep-uniq}
  Under the hypotheses of the previous theorem, let us also suppose that $\gamma$ is linear, i.e.
  there exists $\alpha>0$ such that
  \beq
    \label{gamma_lin}
      \gamma(r)=\alpha r\,, \quad r\in\Ar\,.
  \eeq
  Then, there exists a constant $C>0$, which depends only on $\alpha$, $\Omega$ and $T$, such that
  for every solution $(u_i,v_i,\xi_i)$ of the problem \eqref{thm1}--\eqref{thm3} corresponding to the data $\{u_0^i\,, h_i\,, g_i\}$, for $i=1,2$,
  the following continuous dependence property holds:
  \beq
    \label{dep_cont}
    \begin{split}
    &\l|u_1-u_2\r|^2_{L^2(0,T;\H1)\cap L^\infty(0,T;\L2)}\\
    &\leq C\left[\nL2{u_1^0-u_2^0}^2+
     \l|g_1-g_2\r|_{L^2(0,T;\L2)}^2+\l|h_1-h_2|\r|_{L^2(0,T;L^2(\Gamma_1))}^2\right]
     \end{split}
  \eeq
  In particular, problem \eqref{thm1}--\eqref{thm3} has a unique solution.
\end{thm}


\section{The proof of the first result}
\setcounter{equation}{0}

The idea is to apply Theorem \ref{ben_show}, for $\lambda>0$ fixed, in the spaces $V=H^1(\Omega)$, $W=L^2(\Omega)$, to the operators
$f$, $\mathscr{A}$ and $\mathscr{B}_\lambda$ defined in \eqref{f}, \eqref{A} and \eqref{B_def} respectively.
Note at first that in this case $W$ is a reflexive Banach space and $V$ is a Hilbert space which is dense and compactly embedded in $W$. So,
if we are able to check hypotheses $[A1]$, $[A2]$, $[B1]$ and $[B2]$, then we can apply Theorem \ref{ben_show},
and Lemma \ref{lemma} follows directly.
The objective of this section is therefore to control the hypotheses of the theorem for $f$, $\mathscr{A}$ and $\mathscr{B}_\lambda$.

\subsection{Checking [A1]}

First of all, $\gamma$ is maximal monotone on $\Ar$, so there exists $\hat{\gamma}:\Ar\rarr(-\infty,+\infty]$ convex, proper and lower semicontinuous
such that $\partial\hat{\gamma}=\gamma$: it is not restrictive to suppose that $\hat{\gamma}(0)=0$ (by adding an appropriate constant). Furthermore,
the condition $\gamma(0)\ni0$ implies that $0$ is a minimizer for $\hat{\gamma}$, so we have $\hat{\gamma}:\Ar\rarr[0,+\infty]$.

Let us consider $\phi:L^2(\Omega)\rarr[0,+\infty]$ given by
\beq
  \label{phi}
  \phi(u)=
  \begin{cases}
    \int_\Omega{\hat{\gamma}(u(x))\,dx} \quad &\text{if $\hat{\gamma}(u)\in L^1(\Omega)$}\\
    +\infty \quad &\text{otherwise.}
  \end{cases}
\eeq
\noindent We know from the general theory that $\phi$ is proper, convex and lower semicontinuous on $L^2(\Omega)$, and that $i^*\circ\partial\phi\circ i=\mathscr{A}$,
where $\mathscr{A}$ is defined in \eqref{A}. 

We have to check that $\phi$ is continuous at some point of $H^1(\Omega)$.
Let's first prove that $D(\partial\phi)=L^2(\Omega)$: let $u\in L^2(\Omega)$ and $v:\Omega\rarr\Ar$ be a measurable function such that
$v(x)\in\gamma(u(x))$ almost everywhere in $\Omega$. Then, condition \eqref{lin_bound} implies that
\[
  \left|v(x)\right|\leq C_1\left| u(x)\right|+C_2 \quad \text{for a.e. } x\in\Omega\,;
\]
hence, it turns out that $v\in L^2(\Omega)$, and $u\in D(\partial\phi)$. As a consequence, $D(\phi)=L^2(\Omega)$, since in general $D(\partial\phi)\subseteq D(\phi)$;
furthermore, since each convex, proper and lower semicontinuous function on a Banach space is continuous at the interior of its domain 
(for details see \cite[Prop.~2.12, p.~40]{brezis}),
we conclude that $\phi$ is continuous everywhere in $L^2(\Omega)$. In particular, we have $\phi:L^2(\Omega)\rarr[0,+\infty)$.

Finally, we have to check that $\partial\phi\circ i:H^1(\Omega)\rarr\ L^2(\Omega)$ is bounded. Let $u\in H^1(\Omega)$ and $v\in(\partial\phi\circ i)(u)$:
then, $v(x)\in\gamma(u(x))$ for almost every $x\in\Omega$ and condition \eqref{lin_bound} implies that
\[
  \left|v(x)\right|\leq C_1\left|u(x)\right|+C_2 \quad \text{for a.e. } x\in\Omega\,,
\]
from which, passing to the norms, we deduce the required boundedness.

\subsection{Checking [A2]}

We only have to control that the operator $\partial\phi\circ i: L^2(0,T;H^1(\Omega))\rarr L^2(0,T;L^2(\Omega))$ is bounded
(more precisely, we are considering the operator induced by $\partial\phi\circ i$ on the time-dependent spaces by the a.e. relation).
Let $u\in L^2(0,T;H^1(\Omega))$ and $v\in(\partial\phi\circ i)(u)$: then, $u(t)\in H^1(\Omega)$
and $v(t)\in\partial\phi(u(t))$ almost everywhere in $(0,T)$, so, as we have observed above, we have
\[
  \left|v(t)(x)\right|\leq C_1\left|u(t)(x)\right|+C_2 \quad \text{for a.e. } (x,t)\in\Omega\times(0,T)\,,
\]
which easily implies the required boundedness.

\subsection{Checking [B1]}

We now have to control that $\mathscr{B}_\lambda:\H1\rarr\H1'$ is maximal monotone and bounded: let us start from the last property.
Recall that for all $\zeta\in\H1$, the trace of $\zeta$ on $\Gamma_1$ is an element of $\gL2$ and 
\beq
\label{stimh1}
  \ngL2\zeta\leq C\nH1\zeta \qquad\text{for some constant C>0 (independent of $\zeta$)\,;}
\eeq
furthermore, $\beta_\lambda$ is $\frac{1}{\lambda}$-Lipschitz continuous on $\gL2$ and $\beta_\lambda(0)=0$, then the following holds:
\[
  \ngL2{\beta_\lambda(\zeta)}\leq\frac{1}{\lambda}\ngL2\zeta\,.
\]
Hence, if $u\in H^1(\Omega)$, then for all $z\in H^1(\Omega)$, taking \eqref{B_def} an \eqref{stimh1} into account, we have
\[
  \begin{split}
    \left|\left<\mathscr{B}_\lambda(u),z\right>\right|&\leq \lambda\nL2u\nL2z+\nL2{\nabla u}\nL2{\nabla z}+\ngL2{\beta_\lambda(u)}\ngL2z\\
    &\leq\left[\max\{\lambda,1\}+\frac{C^2}{\lambda}\right]\nH1u\nH1z \qquad \forall z\in\H1\,,
  \end{split}
\]
from which we obtain
\beq
\label{bound_Blam}
  \l|\mathscr{B}_\lambda(u)\r|_{H^1(\Omega)'}\leq\left[\max\{\lambda,1\}+\frac{C^2}{\lambda}\right]\nH1u\ \quad \forall u\in H^1(\Omega)\,,
\eeq
that is our required boundedness on $\mathscr{B}_\lambda$.

We now have to control that $\mathscr{B}_\lambda$ is maximal monotone: in this sense, we check it using a characterization 
of the maximal monotonicity, i.e. we show that $Rg(\mathscr{R}+\mathscr{B}_\lambda)=\H1'$, where $\mathscr{R}:\H1\rarr\H1'$ is the usual Riesz operator.
For any given $F\in H^1(\Omega)'$, we have to find $u\in \H1$ (which will depend a posteriori on $\lambda$, of course) such that
\[
  \int_\Omega{uz\,dx}+\int_\Omega{\nabla u\cdot\nabla z\,dx}+\left<\mathscr{B}_\lambda(u),z\right>=\left<F,z\right> \quad \forall z\in\H1\,,
\]
or in other words that
\beq
\label{maxmon}
  \left(1+\lambda\right)\int_\Omega{uz\,dx}+2\int_\Omega{\nabla u\cdot\nabla z\,dx}+\int_{\Gamma_1}{\beta_\lambda(u)z\,ds}=\left<F,z\right> \quad \forall z\in\H1\,.
\eeq
First of all, we introduce $\beta_\lambda^\epsilon$ as
\beq
  \label{tronc}
  \beta_\lambda^\epsilon(r)=
  \begin{cases}
    \beta_\lambda(r) \quad &\text{if $\quad\left|\beta_\lambda(r)\right|\leq\frac{1}{\epsilon}$}\\
    \frac{1}{\epsilon} &\text{if $\quad\beta_\lambda(r)>\frac{1}{\epsilon}$}\\
    -\frac{1}{\epsilon} &\text{if $\quad\beta_\lambda(r)<-\frac{1}{\epsilon}$}
  \end{cases}
\eeq
and for a fixed $\epsilon>0$, we look for $u_\epsilon\in\H1$ such that
\beq
  \label{maxmon_eps1}
  \left(1+\lambda\right)\int_\Omega{u_\epsilon z\,dx}+2\int_\Omega{\nabla u_\epsilon\cdot\nabla z\,dx}
  +\int_{\Gamma_1}{\beta_\lambda^\epsilon(u_\epsilon)z\,ds}=\left<F,z\right> \quad \forall z\in\H1\,.
\eeq
The idea is to use a fixed point argument in the following sense: let $\delta\in(0,1/2)$ and $\bar{u}\in H^{1-\delta}(\Omega)$. We now solve
for a fixed $\bar{u}$ the following variational equation:
\beq
  \label{maxmon_eps}
  \left(1+\lambda\right)\int_\Omega{u_\epsilon z\,dx}+2\int_\Omega{\nabla u_\epsilon\cdot\nabla z\,dx}=
  -\int_{\Gamma_1}{\beta_\lambda^\epsilon(\bar{u})z\,ds}+\left<F,z\right> \quad \forall z\in\H1\,.
\eeq
Please note that for such a choice of $\bar{u}$, the trace of $\bar{u}$ on $\Gamma_1$ is in $\gL2$, and everything is thus well defined.
We would like to apply the Lax - Milgram lemma. First of all, note that since the trace of $\bar{u}$ is in $\gL2$ and $\beta_\lambda^\epsilon$
is $\frac{1}{\lambda}$-Lipschitz continuous (because so is $\beta_\lambda$ and thanks to \eqref{tronc}), then also $\beta_\lambda^\epsilon(\bar{u})\in\gL2$,
and thus
\[
  z\mapsto -\int_{\Gamma_1}{\beta_\lambda^\epsilon(\bar{u})z\,ds}+\left<F,z\right>\,,	\quad z\in\H1
\]
is an element of $\H1'$. Furthermore, it is clear that 
\[
  \left(z_1,z_2\right)\mapsto \left(1+\lambda\right)\int_\Omega{z_1z_2\,dx}+2\int_\Omega{\nabla z_1\cdot\nabla z_2\,dx}\,,
  \quad \left(z_1,z_2\right)\in\H1\times\H1
\]
is a bilinear continuous and coercive form on $\H1$. Thus, the Lax - Milgram lemma implies that there exists a unique $u_\epsilon\in\H1$ solving \eqref{maxmon_eps}.
At this point, note that if we use the specific test function $z=u_\epsilon$ in \eqref{maxmon_eps}, owing to the Hölder inequality
we obtain
\[
  \begin{split}
    \left(1+\lambda\right)&\nL2{u_\epsilon}^2+2\nL2{\nabla u_\epsilon}^2\leq
    \int_{\Gamma_1}{\left|\beta_\lambda^\epsilon(\bar{u})\right|\left|u_\epsilon\right|\,ds}+\left|\left<F,u_\epsilon\right>\right|\\
    &\leq \ngL2{\beta_\lambda^\epsilon(\bar{u})}\ngL2{u_\epsilon}+\l|F\r|_{H^1(\Omega)'}\nH1{u_\epsilon}\\
    &\leq\frac{1}{\epsilon}\left|\Gamma_1\right|\ngL2{u_\epsilon}+\l|F\r|_{H^1(\Omega)'}\nH1{u_\epsilon}\leq
    \left[\frac{C}{\epsilon}\left|\Gamma_1\right|+\l|F\r|_{H^1(\Omega)'}\right]\nH1{u_\epsilon}\\
    &\leq\frac{1}{2}\left(\frac{C}{\epsilon}\left|\Gamma_1\right|+\l|F\r|_{H^1(\Omega)'}\right)^2+\frac{1}{2}\nH1{u_\epsilon}^2\,,
  \end{split}
\]
from which we deduce the estimate 
\[
  \left(\frac{1}{2}+\lambda\right)\nL2{u_\epsilon}^2+\frac{3}{2}\nL2{\nabla u_\epsilon}^2\leq
  \frac{1}{2}\left(\frac{C}{\epsilon}\left|\Gamma_1\right|+\l|F\r|_{H^1(\Omega)'}\right)^2\,.
\]
At this point, it is suitable to introduce the convex set
\[
  K_\epsilon:=\left\{z\in\H1: \left(\frac{1}{2}+\lambda\right)\nL2z^2+\frac{3}{2}\nL2{\nabla z}^2\leq
  \frac{1}{2}\left(\frac{C}{\epsilon}\left|\Gamma_1\right|+\l|F\r|_{H^1(\Omega)'}\right)^2 \right\}
\]
and the mapping
\beq
  \label{psi_eps}
  \Psi_\epsilon:K_\epsilon\rarr K_\epsilon\,, \quad \Psi_\epsilon(\bar{u})=u_\epsilon\,, \for \bar{u}\in K_\epsilon\,;
\eeq
so, $\Psi_\epsilon(\bar{u})$ is the unique solution $u_\epsilon$ of problem \eqref{maxmon_eps}, corresponding to $\bar{u}\in K_\epsilon$.
Furthermore, $u_\epsilon$~is a solution of problem \eqref{maxmon_eps1} if and only if $\Phi_\epsilon(u_\epsilon)=u_\epsilon$.

Hence, in order to solve \eqref{maxmon_eps}, we have to find a fixed point of $\Psi_\epsilon$: the idea is to use the Schauder fixed point theorem,
which we briefly recall.
\begin{thm}[Schauder]
  \label{schau}
  Let $X$ be a Banach space, $C$ a compact convex subset of $X$ and $f:C\rarr C$ a continuous function:
  then, there exists $x_0\in C$ such that $f(x_0)=x_0$.
\end{thm}
We want to apply Theorem \ref{schau} to $\Psi_\epsilon$: first of all, we have to choose the Banach space~$X$,
in the notation of the result. In order to obtain the compactness property of~$K$, the idea is to work in $X=H^{1-\delta}(\Omega)$.
In fact, it is clear that $K_\epsilon$ is bounded in $\H1$, and since the inclusion $\H1\hookrightarrow H^{1-\delta}(\Omega)$ is compact,
we have that $K_\epsilon$ is a compact set of $H^{1-\delta}(\Omega)$. We now have to check that $\Psi_\epsilon$ is continuous with respect
to the topology of $H^{1-\delta}(\Omega)$: so, let $\{\bar{u}_n\}_n\subseteq K_\epsilon$, $\bar{u}\in K_\epsilon$, and let us show that
\beq
\label{cont}
  \bar{u}_n\rarr\bar{u} \text{ in $H^{1-\delta}(\Omega)$} \quad\impl\quad \Psi_\epsilon(\bar{u}_n)\rarr\Psi_\epsilon(\bar{u}) \text{ in $H^{1-\delta}(\Omega)$}\,.
\eeq
If we call $u_{\epsilon,n}=\Psi_\epsilon(\bar{u}_n)$ and $u_\epsilon=\Psi_\epsilon(\bar{u})$, then the definition of $\Psi_\epsilon$ itself 
and the difference of the corresponding equations allow us to infer that
\[
  \left(1+\lambda\right)\int_\Omega{(u_{\epsilon,n}-u_\epsilon) z\,dx}+2\int_\Omega{\nabla (u_{\epsilon,n}-u_\epsilon)\cdot\nabla z\,dx}=
  -\int_{\Gamma_1}{\left(\beta_\lambda^\epsilon(\bar{u}_n)-\beta_\lambda^\epsilon(\bar{u})\right)z\,ds}
\]
for all $z\in\H1$;
testing now by $z=u_{\epsilon,n}-u_\epsilon$, using Hölder inequality and \eqref{stimh1}, we arrive at
\[
  \begin{split}
     \left(1+\lambda\right)&\nL2{u_{\epsilon,n}-u_\epsilon}^2+2\nL2{\nabla (u_{\epsilon,n}-u_\epsilon)}^2\\
     &\leq\ngL2{\beta_\lambda^\epsilon(\bar{u}_n)-\beta_\lambda^\epsilon(\bar{u})}\ngL2{u_{\epsilon,n}-u_\epsilon}\\
     &\leq\frac{1}{\lambda}\ngL2{\bar{u}_n-\bar{u}}\ngL2{u_{\epsilon,n}-u_\epsilon}\leq
     \frac{C}{\lambda}\ngL2{\bar{u}_n-\bar{u}}\nH1{u_{\epsilon,n}-u_\epsilon}\\
     &\leq \frac{C^2}{2\lambda^2}\ngL2{\bar{u}_n-\bar{u}}^2+\frac{1}{2}\nH1{u_{\epsilon,n}-u_\epsilon}^2\,.
   \end{split}
\]
Hence, we have obtained that
\[
  \left(\frac{1}{2}+\lambda\right)\nL2{u_{\epsilon,n}-u_\epsilon}^2+\frac{3}{2}\nL2{\nabla (u_{\epsilon,n}-u_\epsilon)}^2
  \leq \frac{C^2}{2\lambda^2}\ngL2{\bar{u}_n-\bar{u}}^2\,;
\]
now, since if $\bar{u}_n\rarr\bar{u}$ in $H^{1-\delta}(\Omega)$ then in particular $\bar{u}_n\rarr\bar{u}$ in $\gL2$
for the traces, the relation above implies \eqref{cont}, and the continuity of $\Psi_\epsilon$ is proven.
So, we are able to apply Theorem \ref{schau} to $\Psi_\epsilon$: we find out that there exists $u_\epsilon\in K_\epsilon$
such that $\Psi_\epsilon(u_\epsilon)=u_\epsilon$, i.e. that there exists a solution $u_\epsilon$ of problem \eqref{maxmon_eps1}.

At this point, we would like to find a solution of problem \eqref{maxmon} taking the limit as $\epsilon\rarr0^+$: in order to do this, we need some estimates
on $u_\epsilon$ independent of $\epsilon$. It is immediate to check that if we test equation \eqref{maxmon_eps1} by $z=u_\epsilon$
(actually, this is an admissible choice of $z$), we obtain
\[
  \left(1+\lambda\right)\int_\Omega{u_\epsilon^2\,dx}+2\int_\Omega{\left|\nabla u_\epsilon\right|^2\,dx}
  +\int_{\Gamma_1}{\beta_\lambda^\epsilon(u_\epsilon)u_\epsilon\,ds}=\left<F,u_\epsilon\right>\,;
\]
since $\beta_\lambda^\epsilon$ is monotone and $0\in\beta(0)$ we deduce that
\[
  \left(1+\lambda\right)\nL2{u_\epsilon}^2+2\nL2{\nabla u_\epsilon}^2\leq\l|F\r|_{H^1(\Omega)'}\nH1{u_\epsilon}\leq\frac{1}{2}\l|F\r|_{H^1(\Omega)'}^2+\frac{1}{2}\nH1{u_\epsilon}^2\,,
\]
from which
\beq
  \label{stim_ueps}
  \left(\frac{1}{2}+\lambda\right)\nL2{u_\epsilon}^2+\frac{3}{2}\nL2{\nabla u_\epsilon}^2\leq\l|F\r|_{H^1(\Omega)'}^2 \quad \forall\epsilon>0\,.
\eeq
Hence, $\{u_\epsilon\}_{\epsilon>0}$ is bounded in $\H1$, and therefore there exists a sequence $\epsilon_n\searrow0$ and $u\in\H1$ such that
\[
  u_{\epsilon_n}\rarrw u \quad\text{in}\quad \H1\,;
\] 
in particular, this condition implies that as $n\rarr\infty$
\[
  u_{\epsilon_n}\rarr u \quad\text{in}\quad H^{1-\delta}(\Omega)\,,
\]
\[
  u_{\epsilon_n}\rarr u \quad\text{in}\quad \gL2\,.
\]
We now want to take the limit in equation \eqref{maxmon_eps1} evaluated for $u_{\epsilon_n}$.
Thanks to the weak convergence of $u_{\epsilon_n}$, we have that
\[
  \int_\Omega{u_{\epsilon_n} z\,dx}\rarr\int_\Omega{u z\,dx} \aand 
  \int_\Omega{\nabla u_{\epsilon_n}\cdot\nabla z\,dx}\rarr\int_\Omega{\nabla u\cdot\nabla z\,dx}\,;
\]
furthermore, the Lipshitz-continuity of $\beta_\lambda^\epsilon$ leads to
\[
  \begin{split}
    \ngL2{\beta_\lambda^\epsilon(u_{\epsilon_n})-\beta_\lambda(u)}&\leq
    \ngL2{\beta_\lambda^\epsilon(u_{\epsilon_n})-\beta_\lambda^\epsilon(u)}+\ngL2{\beta_\lambda^\epsilon(u)-\beta_\lambda(u)}\\
    &\leq\frac{1}{\lambda}\ngL2{u_{\epsilon_n}-u}+\ngL2{\beta_\lambda^\epsilon(u)-\beta_\lambda(u)}\rarr0
  \end{split}
\]
since $u_{\epsilon_n}\rarr u$ in $\gL2$ and thanks to the dominated convergence theorem.
Hence, taking the limit as $n\rarr\infty$ we find exactly that $u$ satisfies equation \eqref{maxmon}: this ends the proof 
of the maximal monotonicity of $\mathscr{B}_\lambda$.

\begin{rmk}
  In order to apply Theorem \ref{ben_show} we need $\mathscr{B}_\lambda$ to be maximal monotone.
  Actually, we can say something more: $\mathscr{B}_\lambda$ is a subdifferential, or, more precisely,
  there exists $\psi_\lambda:\H1\rarr(-\infty,+\infty]$ proper, convex and lower semicontinuous such that
  $\partial\psi_\lambda=\mathscr{B}_\lambda$. In particular, $\psi_\lambda$ has the following expression:
  \[
    \psi_\lambda(z)=
    \frac{\lambda}{2}\int_\Omega{z^2\,dx}+\frac{1}{2}\int_\Omega{\left|\nabla z\right|^2\,dx}+\int_{\Gamma_1}{\hat{\beta}_\lambda(z)\,ds}\,,
  \]
  where $\hat{\beta}_\lambda$ is the proper, convex and continuous function on $\Ar$ such that $\hat{\beta}_\lambda(0)=0$ and
  $\partial\hat{\beta}_\lambda=\beta_\lambda$.
\end{rmk}
  
\subsection{Checking [B2]}

We now have to control that the operator $\mathscr{B}_\lambda:L^2(0,T;\H1)\rarr L^2(0,T;\H1')$ is bounded and coercive.
Let $u\in L^2(0,T;\H1)$ and $v\in L^2(0,T;\H1)$: then, the estimate \eqref{bound_Blam} implies that
\[
  \left|\left<\mathscr{B}_\lambda(u(t)),v(t)\right>\right|\leq\left[\max\{\lambda,1\}+\frac{C^2}{\lambda}\right]\nH1{u(t)}\nH1{v(t)} \for\text{a.e. } t\in(0,T)\,;
\]
integrating the previous expression on $(0,T)$ we obtain
\[
  \int_{(0,T)}{\left|\left<\mathscr{B}_\lambda(u(t)),v(t)\right>\right|\,dt}\leq\left[\max\{\lambda,1\}+\frac{C^2}{\lambda}\right]
  \l|u\r|_{L^2(0,T;\H1)}\l|v\r|_{L^2(0,T;\H1')}\,.
\]
Since this is true for all $v\in L^2(0,T;\H1)$, we have proved the required boundedness on $\mathscr{B}_\lambda$.

We now focus on the coercivity of $\mathscr{B}_\lambda$: for each $u\in L^2(0,T;\H1)$, using the monotonicity of $\beta_\lambda$ we have
\[
  \begin{split}
    \left<\mathscr{B}_\lambda(u(t)),u(t)\right>&=\lambda\int_\Omega{u(t)^2\,dx}+\int_\Omega{\left|\nabla u(t)\right|^2\,dx}
    +\int_{\Gamma_1}{\beta_\lambda(u(t))u(t)\,ds}\geq\\
    &\geq\min\{\lambda,1\}\nH1{u(t)}^2 \for \text{a.e. } t\in(0,T)\,;
  \end{split}
\]
so, integrating we deduce that
\[
  \int_0^T{\left<\mathscr{B}_\lambda(u(t)),u(t)\right>\,dt}\geq\min\{\lambda,1\}\l|u\r|_{L^2(0,T;\H1)}^2\,,
\]
which implies
\[
  \frac{\int_0^T{\left<\mathscr{B}_\lambda(u(t)),u(t)\right>\,dt}}{\l|u\r|_{L^2(0,T;\H1)}}\geq\min\{\lambda,1\}\l|u\r|_{L^2(0,T;\H1)}\rarr+\infty 
\]
if $\l|u\r|_{L^2(0,T;\H1)}\rarr+\infty$,
and also the last hypothesis is satisfied.


\section{The proof of the second result}
\setcounter{equation}{0}

First of all, Theorem \ref{lemma} tells us that for each $\lambda>0$ there exist  $u_\lambda\in L^2(0,T;H^1(\Omega))$
and $v_\lambda\in H^1(0,T;H^1(\Omega)')\cap L^2(0,T;\L2)$ such that conditions \eqref{lem1}--\eqref{lem3} hold. In particular,
\eqref{lem1} can be written as follows:
\beq
  \label{eq_lam}
    \begin{split}
      \left<\frac{\partial v_\lambda}{\partial t}(t),z\right>&+\lambda\int_\Omega{u_\lambda(t)z}+\int_\Omega{\nabla u_\lambda(t)\cdot\nabla z\, dx}+
      \int_{\Gamma_1}{\beta_\lambda(u_\lambda(t))z\,ds}\\
      &=\int_\Omega{g(t)z\,dx}+\int_{\Gamma_1}{h(t)z\,ds}
      \quad \forall z\in H^1(\Omega)\,, \quad \text{for a.e. }t\in(0,T)\,. 
    \end{split}
\eeq
We now want to obtain some estimates on $u_\lambda$ and $\beta_\lambda(u_\lambda)$ independent of $\lambda$, and then 
we will look for a solution of our original problem by taking the limit as $\lambda\rarr0^+$.
In this sense, the argument we are going to rely on needs a higher regularity of $v_\lambda$, i.e.
\beq
  \label{ipbis_v}
  \frac{\partial v_\lambda}{\partial t}\in L^2(0,T;\L2)\,, \quad\text{for } \lambda>0\,,
\eeq
which is not generally ensured. Thus, the idea is to accurately approximate $u_0$ and $h$ with some $\{u_{0,\lambda}\}$ and $\{h_\lambda\}$, in order to gain the required regularity \eqref{ipbis_v}, and then exploit it in developing our argument. Indeed, we show uniform estimates on $u_\lambda$ and $\beta_\lambda(u_\lambda)$
and check that such estimates are independent of both $\lambda$ and the approximations of the data.
In this perspective, we present a first result.
\begin{lem}
  \label{lem_bis}
  If $u_0,v_0\in\H1$ and $h\in H^1(0,T;L^2(\Gamma_1))$, the solution components $u_\lambda$ and $v_\lambda$ of the problem \eqref{lem1}--\eqref{lem3} satisfy
  \beq
    \label{reg_vlam}
    u_\lambda, v_\lambda\in H^1(0,T;\L2)\cap L^\infty(0,T;\H1)\,, \quad\forall\lambda>0\,.
  \eeq
\end{lem}
\begin{proof}
  Let us proceed in a formal way, taking directly $z=\frac{\partial u_\lambda}{\partial t}$ in \eqref{eq_lam} although the regularity of $u_\lambda$
  does not allow so (actually, a rigorous approach would require a further regularization, which is not restrictive if we keep in mind our goal).
  Taking \eqref{bilip} into account, we have that
  \[
  \int_0^t\int_\Omega{\frac{\partial v_\lambda}{\partial t}(r)\frac{\partial u_\lambda}{\partial t}(r)\,dxdr}\geq
  c_\gamma\int_0^t\int_\Omega{\left|\frac{\partial u_\lambda}{\partial t}(r)\right|^2\,dxdr}\,,
  \]
  where $c_\gamma$ is the Lipshitz constant of $\gamma^{-1}$,
  while
  \[
  \begin{split}
  &\int_0^t\int_\Omega{\left(\lambda u_\lambda(r)\frac{\partial u_\lambda}{\partial t}(r)+
  \nabla{u_\lambda}(r)\cdot\nabla{\frac{\partial u_\lambda}{\partial t}(r)}\right)\,dxdr}\\
  &=\frac{\lambda}{2}\int_\Omega{\left| u_\lambda(t)\right|^2\,dx}+\frac{1}{2}\int_\Omega{\left|\nabla u_\lambda(t)\right|^2\,dx}
  -\frac{\lambda}{2}\int_\Omega{\left| u_0\right|^2\,dx}-\frac{1}{2}\int_\Omega{\left|\nabla u_0\right|^2\,dx}\,,
  \end{split}
  \]
  and
  \[
  \int_0^t\int_{\Gamma_1}{\beta_\lambda(u_\lambda(r))\frac{\partial u_\lambda}{\partial t}(r)\,dsdr}=
  \int_{\Gamma_1}{\hat{\beta_\lambda}(u_\lambda(t))\,ds}-\int_{\Gamma_1}{\hat{\beta_\lambda}(u_0)\,ds}\,.
  \]
  Furthermore, thanks to the Young inequality we have that
  \[
  \int_0^t\int_\Omega{g(r)\frac{\partial u_\lambda}{\partial t}(r)\,dxdr}\leq
  \frac{c_\gamma}{2}\int_0^t\int_\Omega{\left|\frac{\partial u_\lambda}{\partial t}(r)\right|^2\,dxdr}+\frac{1}{2c_\gamma}\int_0^t\int_\Omega{\left|g(r)\right|^2\,dxdr}\,,
  \]
  while an integration by parts leads to
  \[
  \int_0^t\int_{\Gamma_1}{h(r)\frac{\partial u_\lambda}{\partial t}(r)\,dsdr}=\int_{\Gamma_1}{h(t)u_\lambda(t)\,ds}
  -\int_{\Gamma_1}{h(0)u_0\,ds}-\int_0^t\int_{\Gamma_1}{\frac{\partial h}{\partial t}(r)u_\lambda(r)\,dsdr}
  \]
  Taking all these considerations into account and using the fact that $\hat{\beta_\lambda}\geq0$, we obtain
  \[
  \begin{split}
  c_\gamma&\int_0^t\int_\Omega{\left|\frac{\partial u_\lambda}{\partial t}(r)\right|^2\,dxdr}+
  \min\left\{\frac{\lambda}{2},\frac{1}{2}\right\}\nH1{u_\lambda(t)}^2\leq
  \frac{\lambda}{2}\int_\Omega{\left| u_0\right|^2\,dx}\\
  &+\frac{1}{2}\int_\Omega{\left|\nabla u_0\right|^2\,dx}
  +\int_{\Gamma_1}{\hat{\beta_\lambda}(u_0)\,ds}+\frac{1}{2c_\gamma}\l|g\r|_{L^2(0,T;\L2)}+\int_{\Gamma_1}{\left|h(0)u_0\right|\,ds}\\
  &+\frac{c_\gamma}{2}\int_0^t\int_\Omega{\left|\frac{\partial u_\lambda}{\partial t}(r)\right|^2\,dxdr}
  +\int_{\Gamma_1}{h(t)u_\lambda(t)\,ds}+\int_0^t\int_{\Gamma_1}{\left|\frac{\partial h}{\partial t}(r)u_\lambda(r)\right|\,dsdr}\,;
  \end{split}
  \]
  owing to the Young inequality and \eqref{stimh1}, since $h\in C^0([0,T];L^2(\Gamma_1))$, we have
  \[
  \int_{\Gamma_1}{h(t)u_\lambda(t)\,ds}\leq \frac{1}{2\epsilon}\l|h\r|_{C^0([0,T];L^2(\Gamma_1))}^2
  +\frac{\epsilon C_\lambda^2}{2}\nH1{u_\lambda(t)}^2
  \]
  and
  \[
  \int_0^t\int_\Omega{\left|\frac{\partial h}{\partial t}(r)u_\lambda(r)\right|\,dsdr}\leq
  \frac{1}{2}\l|\frac{\partial h}{\partial t}\r|_{L^2(0,T;L^2(\Gamma_1))}^2
  +\frac{C_\lambda^2}{2}\int_0^t{\nH1{u_\lambda(r)}^2\,dr}
  \]
  for some constant $C_\lambda>0$ and for all $\epsilon>0$. Thus, if we choose a sufficiently small $\epsilon$ such that
  $\delta:=\min\{1/2,\lambda/2\}-\epsilon C_\lambda^2/2>0$, we deduce that
  \[
  \frac{c_\gamma}{2}\int_0^t\int_\Omega{\left|\frac{\partial u_\lambda}{\partial t}(r)\right|^2\,dxdr}+
  \delta\nH1{u_\lambda(t)}^2\leq
  \left[\ldots\right]+\frac{C_\lambda^2}{2}\int_0^t{\nH1{u_\lambda(r)}^2\,dr}\,,
  \]
  and the Gronwall lemma ensures that
  \[
  \frac{c_\gamma}{2\delta}\int_0^t\int_\Omega{\left|\frac{\partial u_\lambda}{\partial t}(r)\right|^2\,dxdr}+
  \nH1{u_\lambda(t)}^2\leq \frac{1}{\delta}\left[\ldots\right]
  e^{\frac{C_\lambda^2}{2\delta}T} \quad\text{for a.e. } t\in(0,T)\,.
  \]
Hence, $u_\lambda\in H^1(0,T;\L2)\cap L^\infty(0,T;\H1)$ and thanks to \eqref{bilip} and a classical result by Stampacchia (see \cite{mar-miz})
we also have that $v_\lambda\in H^1(0,T;\L2)\cap L^\infty(0,T;\H1)$.
\end{proof}

As we have anticipated, we consider now some approximations $\{u_{0,\lambda}\}$ and $\{h_\lambda\}$ such that
the following conditions hold:
\beq
  \label{approx1}
  \{u_{0,\lambda}\} \subseteq\H1\,,\quad v_{0,\lambda}=\gamma(u_{0,\lambda})\,, \quad u_{0,\lambda}\rarr u_0 \quad\text{in } \L2\,,
\eeq
\beq
  \label{approx1bis}
  \text{there exists } L>0 \text{ such that }\l|\hat{\beta_\lambda}(u_{0,\lambda)}\r|_{L^1(\Omega)}\leq L \quad\forall\lambda>0\,,
\eeq
\beq
  \label{approx2}
  \{h_\lambda\} \subseteq H^1(0,T;L^2(\Gamma_1))\,, \quad h_\lambda\rarr h \quad\text{in } L^2(0,T;L^2(\Gamma_1))\,.
\eeq
Actually, an approximation $\{u_{0,\lambda}\}$ such that \eqref{approx1}, \eqref{approx1bis} hold exists and a formal proof is given in Subsection \ref{choice}.
Now, thanks to Lemma \ref{lem_bis}, the corresponding solutions $u_\lambda$, $v_\lambda$ given by Theorem \ref{lemma} have the regularity \eqref{reg_vlam}.
We are now ready to prove Theorem \ref{theorem}.

\subsection{The estimate on $u_\lambda$}

It is natural to let $z=u_\lambda(t)\in\H1$ in equality \eqref{eq_lam}: we obtain
\beq
  \label{ulam_1}
  \begin{split}
      \int_\Omega{\frac{\partial v_\lambda}{\partial t}(t)u_\lambda(t)\,dx}&+
      \lambda\nL2{u_\lambda(t)}^2+\nL2{\nabla u_\lambda(t)}^2+
      \int_{\Gamma_1}{\beta_\lambda(u_\lambda(t))u_\lambda(t)\,ds}\\
      &=\int_\Omega{g(t)u_\lambda(t)\,dx}+\int_{\Gamma_1}{h_\lambda(t)u_\lambda(t)\,ds} \quad \text{for a.e. }t\in(0,T)\,.
    \end{split}
\eeq
Note that the duality pairing in \eqref{eq_lam} has become a scalar product, thanks to \eqref{ipbis_v}.
Let's analyse the four terms on the left hand side, separately. First of all, note that since $\beta_\lambda$ is monotone and $\beta_\lambda(0)=0$
we have
\[
  \int_{\Gamma_1}{\beta_\lambda(u_\lambda(t))u_\lambda(t)\,ds}\geq0\,;
\]
furthermore, it is immediate to see that
\[
   \lambda\nL2{u_\lambda(t)}^2+\nL2{\nabla u_\lambda(t)}^2\geq\nL2{\nabla u_\lambda(t)}^2\,.
\]
Now, we focus on the the first term of equation \eqref{ulam_1}: in order to treat it, we recall a known result (for details, see \cite[Lemma~3.3, p.~72]{brezis}).
\begin{prop}
  \label{der_conv}
  Let $H$ be a Hilbert space and $\psi:H\rarr(-\infty.+\infty]$ a proper convex and lower semicontinuous function; then, for all
  $v\in H^1(0,T;H)$ and $u\in L^2(0,T;H)$ such that $(v(t),u(t))\in\partial\psi$ for a.e. $t\in[0,T]$, the function
  $t\mapsto\psi(v(t))$ is absolutely continuous on $[0,T]$ and
  \beq
    \frac{d}{dt}\psi(v(t))=\left(w,v'(t)\right)_H \quad \forall w\in\partial\psi(v(t))\,, \for \text{a.e. } t\in[0,T]\,.
  \eeq
\end{prop}

In our specific case, we know that $v_\lambda(t)\in\partial\phi(u_\lambda(t))$ a.e. on $(0,T)$, or equivalently that
$u_\lambda(t)\in(\partial\phi)^{-1}(v_\lambda(t))$. If we introduce the convex conjugate of $\phi$, defined as
\beq
  \label{con}
  \phi^*:\L2\rarr(-\infty,+\infty]\,, \quad \phi^*(z)=\sup_{y\in\L2}\{(z,y)_{\L2}-\phi(y)\}\,,
\eeq
from the general theory we know that the following relation holds:
\beq
  \label{sub_inv}
  \left(\partial\phi\right)^{-1}=\partial\phi^*\,.
\eeq
Hence, we have that $u_\lambda(t)\in\partial\phi^*(v_\lambda(t))$ for a.e. $t\in[0,T]$ and it is natural to apply
Proposition \ref{der_conv} with the choices (in the notations of the proposition) $H=\L2$ and $\psi=\phi^*$.
Then, Proposition \ref{der_conv} tells us that
\[
  \int_\Omega{\frac{\partial v_\lambda}{\partial t}(t)u_\lambda(t)\,dx}=\frac{d}{dt}\phi^*(v_\lambda(t)) \for \text{a.e. } t\in[0,T]\,.
\]
Taking these remarks into account, from \eqref{ulam_1} we deduce that
\[
  \frac{d}{dt}\phi^*(v_\lambda(t))+\nL2{\nabla u_\lambda(t)}^2\leq
  \int_\Omega{g(t)u_\lambda(t)\,dx}+\int_{\Gamma_1}{h_\lambda(t)u_\lambda(t)\,ds} \quad \text{for a.e. }t\in(0,T)\,;
\]
furthermore, using the Hölder inequality and equation \eqref{stimh1}, the right hand side can be estimated by
\[
  \begin{split}
    \int_\Omega{g(t)u_\lambda(t)\,dx}&+\int_{\Gamma_1}{h_\lambda(t)u_\lambda(t)\,ds} \\
    &\leq\nL2{g(t)}\nL2{u_\lambda(t)}+\ngL2{h_\lambda(t)}\ngL2{u_\lambda(t)} \\
    &\leq\left(\nL2{g(t)}+C\ngL2{h_\lambda(t)}\right)\nH1{u_\lambda(t)}\,.
  \end{split}
\]
Hence, by integrating with respect to time, we easily obtain
\beq
  \label{ulam_2}
  \begin{split}
    \phi^*(v_\lambda(t))&+\l|\nabla u_\lambda\r|^2_{L^2(0,t;\L2)}\leq \\
    &\leq\phi^*(v_{0,\lambda})+\int_{(0,t)}{\left(\nL2{g(s)}+C\ngL2{h_\lambda(s)}\right)\nH1{u_\lambda(s)}\,ds}\,.
  \end{split}
\eeq
Please note that conditions \eqref{approx1}, \eqref{bilip} and \eqref{approx2} imply
\beq
\label{lim_1}
  \phi^*(v_{0,\lambda})=\int_\Omega{v_{0,\lambda}u_{0,\lambda}\,dx}-\phi(u_{0,\lambda})\leq\nL2{v_{0,\lambda}}\nL2{u_{0,\lambda}}\leq M_1
  \quad\forall\lambda>0
\eeq
\beq
\label{lim_2}
   \l|h_\lambda\r|_{L^2(0,T;\gL2)}\leq M_2 \quad\forall\lambda>0
\eeq
for some positive constants $M_1$ and $M_2$, independent of $\lambda$.

We would like now to find an estimate from below of the term $\phi^*(v_\lambda(t))$: at this purpose, 
let $c_\gamma$ and $C_\gamma$ be the Lipschitz constants of $\gamma^{-1}$ and $\gamma$ respectively.
Then, if $x\in\Ar$, since $\partial\hat{\gamma}=\gamma$ we have
\[
  \hat{\gamma}(x)+\gamma(x)(z-x)\leq\hat{\gamma}(z) \quad \forall z\in\Ar\,.
\]
Making the particular choice $z=0$, taking into account that $\hat{\gamma}(0)=0$ and $\gamma(0)=0$, we~have
\[
  \hat{\gamma}(x)\leq \gamma(x)x\leq\left|\gamma(x)\right|\left|x\right|\leq C_\gamma\left|x\right|^2\,.
\]
At this point, note also that, if we call $\eta(t)=C_\gamma t^2$ for $t\in\Ar$,
then we have
\[
    \hat{\gamma}^*(y)=\sup_{z\in\Ar}\{zy-\hat{\gamma}(z)\}\geq\sup_{z\in\Ar}\{zy-C_\gamma z^2\}=
    \eta^*(y)=\frac{y^2}{4C_\gamma}\,,
\]
since it is easy to check (using the definition of conjugate function) that $\eta^*(y)=\frac{y^2}{4C_\gamma}$:
hence, we reach at
\[
   \hat{\gamma}^*(\gamma(x))\geq\frac{\left|\gamma(x)\right|^2}{4C_\gamma}\geq
   \frac{c_\gamma^2}{4C_\gamma}\left|x\right|^2\,.
\]
 In particular, this estimate implies that there exists $C_1>0$ such that
for all $u\in\L2$ and $v\in\L2$ such that $v\in\gamma(u)$ a.e. in $\Omega$ we have
\[
  \phi^*(v)\geq C_1\nL2u^2\,.
\]
In our specific case, $v_\lambda(t)\in\gamma(u_\lambda(t))$ almost everywhere, whence
\[
  \phi^*(v_\lambda(t))\geq{C_1}\nL2{u_\lambda(t)}^2\,.
\]
Then, from \eqref{ulam_2} it follows that
\[
  \begin{split}
    &{C_1}\l|u_\lambda(t)\r|_{\L2}^2+\l|\nabla u_\lambda\r|_{L^2(0,t;\L2)}^2\\
    &\leq M_1+\frac{1}{2}\int_0^t{\left(\nL2{g(s)}+C\ngL2{h_\lambda(s)}\right)^2\,ds}+
    \frac{1}{2}\int_0^t{\nH1{u_\lambda(s)}^2\,ds}\,, \\
  \end{split}
\]
which easily implies that for all $t\in(0,T)$
\[
  \begin{split}
    &{C_1}\l|u_\lambda(t)\r|_{\L2}^2+\frac{1}{2}\l|\nabla u_\lambda\r|_{L^2(0,t;\L2)}^2\\
    &\leq\left[M_1+\l|g\r|_{L^2(0,T;\L2)}^2+C^2\l|h_\lambda\r|_{L^2(0,T;\gL2)}^2\right]+
    \frac{1}{2}\int_0^t{\nL2{u_\lambda(s)}^2\,ds}\,.
  \end{split}
\]
Please note that condition \eqref{lim_2} ensures the existence of a positive constant $C_2$, independent of $\lambda$, such that
\[
  \left[\ldots\right]\leq C_2 \quad\forall\lambda>0\,.
\]
In particular, we have
\[
  {C_1}\l|u_\lambda(t)\r|_{\L2}^2\leq C_2+ \frac{1}{2}\int_0^t{\nL2{u_\lambda(s)}^2\,ds}
\]
and the Gronwall lemma ensures that
\[
   {C_1}\l|u_\lambda(t)\r|_{\L2}^2\leq C_2 e^{t/2}\leq C_2 e^{T/2} \quad\forall t\in(0,T)\,.
\]
Hence, we have found that there exists a positive constant $A_1>0$, independent of $\lambda$, such that
\beq
  \label{ulam_3}
  \l|u_\lambda\r|_{L^\infty(0,T;\L2)}\leq A_1 \quad \forall\lambda>0\,.
\eeq
Furthermore, replacing \eqref{ulam_3} in our last inequality it follows that there exists also $A_2>0$, independent of $\lambda$, such that
\beq
  \label{ulam_4}
  \l|u_\lambda\r|_{L^2(0,T;\H1)}\leq A_2 \quad \forall\lambda>0\,,
\eeq
which easily leads, thanks to \eqref{bilip}, to
\beq
  \label{ulam_5}
   \l|v_\lambda\r|_{L^2(0,T;\H1)}\leq A_3 \quad \forall\lambda>0\,,
\eeq
for a positive constant $A_3$, independent of $\lambda$ (note the connection with Remark \ref{rem}).

\subsection{The estimate on $\beta_\lambda(u_\lambda)$}

The idea is now to test equation \eqref{eq_lam} by $z=\beta_\lambda(u_\lambda(t))$: firstly, we have to control
that this is an admissible choice, or in other words that $\beta_\lambda(u_\lambda(t))\in~\H1$. Since $u_\lambda(t)\in\L2$,
$\beta_\lambda$ is $\frac{1}{\lambda}$-lipschitz continuous and $\beta_\lambda(0)=0$, it follows that $\beta_\lambda(u_\lambda(t))\in\L2$.
Furthermore, thanks to the Lipschitz continuity as well, we also have that $\beta_\lambda(u_\lambda(t))\in~\H1$.

Testing now \eqref{eq_lam} by $z=\beta_\lambda(u_\lambda(t))$ we obtain
\beq
  \label{blam_1}
  \begin{split}
      &\int_\Omega{\frac{\partial v_\lambda}{\partial t}(t)\beta_\lambda(u_\lambda(t))\,dx}
      +\lambda\int_\Omega{u_\lambda(t)\beta_\lambda(u_\lambda(t))\,dx}
      +\int_\Omega{\nabla u_\lambda(t)\cdot\nabla\beta_\lambda(u_\lambda(t))\, dx}\\
      &+\ngL2{\beta_\lambda(u_\lambda(t))}^2
      =\int_\Omega{g(t)\beta_\lambda(u_\lambda(t))\,dx}+\int_{\Gamma_1}{h_\lambda(t)\beta_\lambda(u_\lambda(t))\,ds}\,.
    \end{split}
\eeq

Let's handle the different terms of \eqref{blam_1} separately: thanks to the monotonicity of $\beta_\lambda$
and the fact that $\beta_\lambda(0)=0$, we have
\[
  \lambda\int_\Omega{u_\lambda(t)\beta_\lambda(u_\lambda(t))\,dx}\geq0\,,
\]
while the monotonicity of $\beta_\lambda$ implies that
\[
  \int_\Omega{\nabla u_\lambda(t)\cdot\nabla\beta_\lambda(u_\lambda(t))\, dx}\geq0\,.
\]
Let us focus on the first term: integrating with respect to time we have
\[
  \begin{split}
    &\int_0^t\int_\Omega{\frac{\partial v_\lambda}{\partial s}(s)\beta_\lambda(u_\lambda(s))\,dx\,ds}=
    \int_0^t\int_\Omega{\frac{\partial\gamma(u_\lambda)}{\partial s}(s)\beta_\lambda(u_\lambda(s))\,dx\,ds}\\
    &=\int_0^t\int_\Omega{\gamma'(u_\lambda(s))u_\lambda'(s)\beta_\lambda(u_\lambda(s))\,dx\,ds}=
    \int_\Omega{j_\lambda(u_\lambda(t))\,dx}-\int_\Omega{j_\lambda(u_{0,\lambda})\,dx}\,,
  \end{split}
\]
where
\[
  j_\lambda(r):=\int_0^r{\gamma'(s)\beta_\lambda(s)\,ds}\,, \quad r\in\Ar\,.
\]
Now, thanks to \eqref{bilip}, if we let $c_\gamma$ be the Lipschitz-constant of $\gamma^{-1}$, as usual,  we have
\[
  \int_0^t\int_\Omega{\frac{\partial v_\lambda}{\partial s}(s)\beta_\lambda(u_\lambda(s))\,dx\,ds}\geq
  c_\gamma\int_\Omega{\hat{\beta_\lambda}(u_\lambda(t))\,dx}-\int_\Omega{j_\lambda(u_{0,\lambda})\,dx}\,.
\]
Taking all these remarks into account, from equation \eqref{blam_1} we obtain
\beq
  \label{blam_2}
  \begin{split}
    c_\gamma\int_\Omega{\hat{\beta_\lambda}(u_\lambda(t))\,dx}&+\l|\beta_\lambda(u_\lambda)\r|^2_{L^2(0,t;\gL2)}\leq
    \int_\Omega{j_\lambda(u_{0,\lambda})\,dx}\\
    &+\int_0^t{\left[\int_\Omega{g(r)\beta_\lambda(u_\lambda(r))\,dx}
    +\int_{\Gamma_1}{h_\lambda(r)\beta_\lambda(u_\lambda(r))\,ds}\right]\,dr}\,.
  \end{split}
\eeq
Please note that hypotheses \eqref{ip_beta} and \eqref{ip_g} imply that
\[
  \int_0^t{\int_\Omega{g(r)\beta_\lambda(u_\lambda(r))\,dx}\,dr}\leq
  \int_0^t{\l|g(r)\r|_{L^\infty(\Omega)}\left(D_1\int_\Omega{\hat{\beta_\lambda}(u_\lambda(r))\,dx}+D_2\left|\Omega\right|\right)\,dr}\,,
\]
while thanks to the Young inequality we have
\[
  \int_0^t{\int_{\Gamma_1}{h_\lambda(r)\beta_\lambda(u_\lambda(r))\,ds}\,dr}\leq
  \frac{1}{2}\l|h_\lambda\r|_{L^2(0,t;\gL2)}^2+\frac{1}{2}\int_0^t\int_\Omega{\left|\beta_\lambda(u_\lambda(r))\right|^2\,dx\,dr}\,;
\]
substituting in \eqref{blam_2} we obtain
\beq
  \label{blam_3}
  \begin{split}
    c_\gamma\int_\Omega{\hat{\beta_\lambda}(u_\lambda(t))\,dx}&+\frac{1}{2}\l|\beta_\lambda(u_\lambda)\r|^2_{L^2(0,t;\gL2)}\leq
    \int_\Omega{j_\lambda(u_{0,\lambda})\,dx}\\
    &+\frac{1}{2}\l|h_\lambda\r|_{L^2(0,T;\gL2)}^2+D_2\left|\Omega\right|\l|g\r|_{L^1(0,T;L^\infty(\Omega))}\\
    &+D_1\int_0^t{\l|g(r)\r|_{L^\infty(\Omega)}\int_\Omega{\hat{\beta_\lambda}(u_\lambda(r))\,dx}\,dr}\,.
  \end{split}
\eeq
At this point, if $C_\gamma$ is the Lipschitz-constant of $\gamma$, property \eqref{prop_yos} ensures that
\[
  j_\lambda(r)=\int_0^r{\gamma'(s)\beta_\lambda(s)\,ds}\leq C_\gamma\hat{\beta_\lambda}(r)\leq C_\gamma\hat{\beta}(r)\,,
\]
and consequently, thanks to \eqref{approx1bis} and \eqref{lim_2}, equation \eqref{blam_3} implies that
\[
  \begin{split}
    \int_\Omega{\hat{\beta_\lambda}(u_\lambda(t))\,dx}&\leq
    \frac{1}{c_\gamma}\left[C_\gamma L+
    \frac{M_2^2}{2}+
    D_2\left|\Omega\right|\l|g\r|_{L^1(0,T;L^\infty(\Omega))}\right]\\
    &+\frac{D_1}{c_\gamma}\int_0^t{\l|g(r)\r|_{L^\infty(\Omega)}\int_\Omega{\hat{\beta_\lambda}(u_\lambda(r))\,dx}\,dr}\,.
  \end{split}
\]
Thanks to the Gronwall lemma, we deduce that
\[
  \int_\Omega{\hat{\beta_\lambda}(u_\lambda(t))\,dx}\leq
  \frac{1}{c_\gamma}\left[\ldots\right]\exp\left(\frac{D_1}{c_\gamma}\int_0^t{\l|g(r)\r|_{L^\infty(\Omega)}\,dr}\right)\leq
  \frac{1}{c_\gamma}\left[\ldots\right]e^{\frac{D_1}{c_\gamma}\l|g\r|_{L^1(0,T;L^\infty(\Omega))}}\,;
\]
hence, there exists $B_1>0$ such that
\beq
  \label{blam_4}
  \l|\hat{\beta_\lambda}(u_\lambda)\r|_{L^{\infty}(0,T;L^1(\Omega))}\leq B_1 \quad \forall\lambda>0\,.
\eeq
Taking this estimate into account in \eqref{blam_3}, it immediately follows that there is $B_2>0$ such that
\beq
  \label{blam_5}
  \l|\beta_\lambda(u_\lambda)\r|_{L^2(0,T;\gL2)}\leq B_2\quad \forall \lambda>0\,.
\eeq

\subsection{The passage to the limit}

Now, we are concerned with passing to the limit as $\lambda\rarr0^+$ in equation \eqref{eq_lam}. We recall the following result (see \cite[Cor.~4, p.~85]{simon}),
which we are going to use next.
\begin{prop}
  \label{simon}
  Let $X\subseteq B\subseteq Y$ be Banach spaces with compact embedding $X\hookrightarrow B$ and
  let $F\subseteq L^p(0,T;X)$ be a bounded set such that $\partial F/\partial t:=\{\partial f/\partial t: f\in F\}$ 
  is bounded in $L^1(0,T;Y)$ for a given $p\geq1$. Then, $F$ is relatively compact in $L^p\left(0,T;B\right)$.
\end{prop}
We would like to apply Proposition \ref{simon} with the choices $X=\H1$, $B=H^{1-\delta}(\Omega)$ ($\delta\in(0,1/2)$), $Y=\H1'$,
 $p=2$ and $F=\{v_\lambda\}_{\lambda>0}$
in order to claim that $\{v_\lambda\}_{\lambda>0}$ is bounded in $L^2\left(0,T;H^{1-\delta}(\Omega)\right)$.
In fact, $F$ is bounded thanks to \eqref{ulam_5}; furthermore,
by comparison in equation \eqref{eq_lam}, using conditions \eqref{ulam_4} and \eqref{blam_5}, we find out that
there exists a constant $E>0$, independent of $\lambda$, such that
\[
  \l|\frac{\partial v_\lambda}{\partial t}\r|_{L^2(0,T;\H1')}\leq E \quad \forall\lambda>0\,.
\]
By weak compactness, we infer that 
\beq
  \label{inpiù}
  v_{\lambda_n}\rarrw v \quad\text{in } H^1(0,T;\H1')\cap L^2(0,T;\H1)
\eeq
for a subsequence $\lambda_n\searrow0$. Moreover,
Proposition \ref{simon} holds and it is a standard matter to obtain
\beq
  \label{lim_v}
  v_{\lambda_n}\rarr v \quad\text{in}\quad L^2\left(0,T;H^{1-\delta}(\Omega)\right)\,, \as n\rarr\infty\,.
\eeq
Let now $u_{\lambda_n}=\gamma^{-1}(v_{\lambda_n})$: then, since $\gamma^{-1}$ is Lipschitz continuous (see \eqref{bilip}),
condition \eqref{lim_v} implies that
\beq
  \label{lim_u1}
  u_{\lambda_n}\rarr\gamma^{-1}(v) \quad\text{in}\quad L^2\left(0,T;\L2\right)\,, \as n\rarr\infty\,.
\eeq
Furthermore, \eqref{ulam_4} tells us that there is a subsequence $\lambda_{n_k}\searrow0$ and $u\in L^2\left(0,T;\H1\right)$ such that
\beq
  \label{lim_u2}
  u_{\lambda_{n_k}}\rarrw u \quad\text{in}\quad L^2\left(0,T;\H1\right)\,, \as k\rarr\infty\,.
\eeq
Conditions \eqref{lim_u1} and \eqref{lim_u2} imply that
\beq
  \label{2nd}
  u=\gamma^{-1}(v) \aand v=\gamma(u) \quad\text{a.e. in } (0,T)\times\Omega
\eeq
and the convergence in \eqref{lim_u2} holds for the entire subsequence $\lambda_n$.
At this point, we recall a general result which is useful to us (see \cite[Chapter~1]{lions-mag}).
\begin{prop}
  For all $\delta\in(0,1)$, there is $\alpha_\delta\in(0,1)$ such that
  \beq
    \label{interp}
    \l|z\r|_{H^{1-\delta}(\Omega)}\leq\nL2z^{\alpha_\delta}\nH1z^{1-\alpha_\delta}\,, \quad \forall z\in\H1\,.
  \eeq
\end{prop}
\noindent Now, thanks to estimate \eqref{interp} and conditions \eqref{lim_u1}-\eqref{lim_u2},
 we have that as $k\rarr\infty$
\[
  \begin{split}
  \int_0^T{\l|u_{\lambda_{n_k}}(t)-u(t)\r|^2_{H^{1-\delta}(\Omega)}\,dt}
  &\leq\int_0^T{\nL2{u_{\lambda_{n_k}}(t)-u(t)}^{2\alpha_\delta}\nH1{u_{\lambda_{n_k}}(t)-u(t)}^{2-2\alpha_\delta}\,dt}\\
  &\leq\l|u_{\lambda_{n_k}}-u\r|_{L^2(0,T;\L2)}^{2\alpha_\delta}\l|u_{\lambda_{n_k}}-u\r|_{L^2(0,T;\H1)}^{2(1-\alpha_\delta)}\rarr0\,;
  \end{split}
\]
it follows that
\beq
  \label{lim_u3}
  u_{\lambda_{n_k}}\rarr u \quad\text{in}\quad L^2\left(0,T;H^{1-\delta}(\Omega)\right)\,, \as k\rarr\infty\,.
\eeq
Furthermore, \eqref{lim_u3} implies the convergence for the traces on $\Gamma_1$:
\beq
  \label{lim_u4}
  u_{\lambda_{n_k}}\rarr u \quad\text{in}\quad L^2\left(0,T;\gL2\right)\,, \as k\rarr\infty\,.
\eeq

Let us focus now on $\beta_\lambda(u_\lambda)$: first of all, note that condition \eqref{blam_5} tells us that
there is $\xi\in L^2\left(0,T;\gL2\right)$ such that (possibly considering another subsequence)
\beq
  \label{lim_b1}
  \beta_{\lambda_{n_k}}\left(u_{\lambda_{n_k}}\right)\rarrw\xi \quad\text{in}\quad L^2\left(0,T;\gL2\right)\,, \as k\rarr\infty\,.
\eeq
Hence, since $\beta$ is maximal monotone, the result stated in \cite[Prop.~1.1, p.~42]{barbu} and the conditions \eqref{lim_u4} and \eqref{lim_b1} ensure that
\beq
  \label{1st}
  u\in D(\beta) \aand \xi\in\beta(u) \quad\text{a.e. in } (0,T)\times\Gamma_1\,.
\eeq

We are now almost ready to pass to the limit as $k\rarr\infty$ and complete the proof. Let's recall equation \eqref{eq_lam},
evaluated for $\lambda_{n_k}$, and argue separately on
the different terms as $k\rarr\infty$:
\begin{gather}
  \int_\Omega{\frac{\partial v_{\lambda_{n_k}}(t)}{\partial t}z\,dx}\rarr\left<\frac{\partial v}{\partial t}(t),z\right>
  \quad\text{thanks to \eqref{inpiù}}\\
  {\lambda_{n_k}} \int_\Omega{u_{\lambda_{n_k}}(t)z}\rarr0 \quad\text{thanks to \eqref{lim_u2}}\\
  \int_\Omega{\nabla u_{\lambda_{n_k}}(t)\cdot\nabla z\, dx}\rarr
  \int_\Omega{\nabla u(t)\cdot\nabla z\, dx} \quad\text{thanks to \eqref{lim_u2}}\\
  \int_{\Gamma_1}{\beta_{\lambda_{n_k}}(u_{\lambda_{n_k}}(t))z\,ds}\rarr
  \int_{\Gamma_1}{\xi(t)z\,ds} \quad\text{thanks to \eqref{lim_b1}}\,.
\end{gather}
Hence, passing to the limit as $k\rarr\infty$ in \eqref{eq_lam} we obtain exactly the thesis \eqref{thm1};
furthermore, conditions \eqref{1st} and \eqref{2nd} yield \eqref{thm2}. Finally, \eqref{thm3}
easily follows from \eqref{approx1} and \eqref{init_pair2}. This finishes the proof.

\subsection{The existence of an approximation $u_{0,\lambda}$}
\label{choice}

As we have anticipated, we now want to prove the existence of
an approximation $\{u_{0,\lambda}\}$ such that conditions \eqref{approx1} and \eqref{approx1bis} hold.

For $\lambda>0$, we define $u_\lambda$ as the solution of the following elliptic problem:
\beq
\label{pr_ell}
  \begin{cases}
  u_{0,\lambda}-\lambda\Delta u_{0,\lambda}=u_0 \quad &\text{in $\Omega$}\,,\\
  \frac{\partial u_{0,\lambda}}{\partial n}=0 \quad &\text{on $\Gamma$}\,.
  \end{cases}
\eeq
Actually, a variational formulation of \eqref{pr_ell} is
\beq
\label{pr_ell2}
  \int_\Omega{u_{0,\lambda}z\,dx}+\lambda\int_\Omega{\nabla u_{0,\lambda}\cdot\nabla z\,dx}=
  \int_{\Omega}{u_0 z\,dx} \quad\forall z\in\H1
\eeq
and a direct application of the Lax-Milgram lemma tells that such $u_{0,\lambda}$ exists and is unique in $\H1$.
Now,
it is easy to check that
\beq
  \label{conv1}
  u_{0,\lambda}\rarr u_0 \quad\text{in } \L2\,.
\eeq
Indeed, if we test equation \eqref{pr_ell2} by $z=u_{0,\lambda}$, using the Young inequality we obtain
\[
  \frac{1}{2}\nL2{u_{0,\lambda}}^2+\lambda\nL2{\nabla u_{0,\lambda}}^2\leq\frac{1}{2}\nL2{u_0}^2\,;
\]
hence, there exists $\widetilde{u_0}\in\L2$ and a subsequence $\{u_{0,\lambda_k}\}_{k\in\En}$ such that
\beq
  \label{conv2}
  u_{0,\lambda_k}\rarrw\widetilde{u_0} \quad\text{in $\L2$} \quad \text{as } k\rarr\infty\,,
\eeq
and
\beq
  \label{conv3}
  \lambda u_{0,\lambda_k}\rarr 0 \quad\text{in } \H1 \quad\text{as } k\rarr\infty\,.
\eeq
Taking \eqref{conv2} and \eqref{conv3} into account and letting $k\rarr\infty$ in equation \eqref{pr_ell2}, we have
\[
  \int_{\Omega}{\widetilde{u_0}z\,dx}=\int_\Omega{u_0 z\,dx} \quad\forall z\in\H1
\]
and we can conclude that $\widetilde{u_0}=u_0$ since $\H1$ is dense in $\L2$. 
Then, the identification of the weak limit implies that the entire family $\{u_{0,\lambda}\}$ weakly converges to $u_0$.
Moreover, as we have
\[
  \limsup_{\lambda\searrow0}\int_\Omega{\left|u_{0,\lambda}\right|^2\,dx}\leq\int_\Omega{\left|u_0\right|^2\,dx}\,,
\]
it turns out that \eqref{conv1} holds.
Thus, condition \eqref{approx1} is satisfied for such a choice of $u_{0,\lambda}$ (and clearly $v_{0,\lambda}=\gamma(u_{0,\lambda})$);
we now check that also \eqref{approx1bis} is satisfied.

Let $z=\beta_\lambda(u_{0,\lambda})$ in \eqref{pr_ell2} ($z\in\H1$ since $\beta_\lambda$
is Lipshitz continuous): taking into account the monotonicity of $\beta_\lambda$, we obtain
\beq
\label{disug}
  0\leq\lambda\int_\Omega{\nabla\beta_\lambda(u_{0,\lambda})\cdot\nabla u_{0,\lambda}\,dx}=
  \int_\Omega{\left(u_0-u_{0,\lambda}\right)\beta_\lambda(u_{0,\lambda})\,dx}\,.
\eeq
Moreover, the definition of subdifferential tells us that
\[
  \hat{\beta_\lambda}(u_{0,\lambda})\leq\hat{\beta_\lambda}(u_0)+\left(u_{0,\lambda}-u_0\right)\beta_\lambda(u_{0,\lambda}) \quad\text{a.e. in } \Omega\,;
\]
then, integrating on $\Omega$ and taking \eqref{disug} and \eqref{prop_yos} into account, we have
\[
  \int_\Omega{\hat{\beta_\lambda}(u_{0,\lambda})\,dx}\leq\int_\Omega{\hat{\beta_\lambda}(u_0)\,dx}+
  \int_\Omega{\left(u_{0,\lambda}-u_0\right)\beta_\lambda(u_{0,\lambda})\,dx}\leq
  \int_\Omega{\hat{\beta_\lambda}(u_0)\,dx}\leq\int_\Omega{\hat{\beta}(u_0)\,dx}\,.
\]
Thus, condition \eqref{approx1bis} is satisfied with the choice
\[
  L=\l|\hat{\beta}(u_0)\r|_{L^1(\Omega)}\,.
\]


\section{Continuous dependence and uniqueness}
\setcounter{equation}{0}

In this last section, we aim at proving Theorem \ref{dep-uniq}, which ensures the continuous dependence of the solutions from the data in problem \eqref{thm1}--\eqref{thm3}.
The most significant assumption in this case is the linearity of $\gamma$: if this is not true, a direct result of uniqueness or continuous dependence
is not evident. Thus, let us assume \eqref{gamma_lin}
and consider two sets of data, $\{u_1^0\,, h_1\,, g_1\}$ and $\{u_2^0\,, h_2\,, g_2\}$, where in particular
\beq
  \label{data}
  u_1^0\,, u_2^0\in \L2\,, \quad h_1\,, h_2 \in L^2(0,T;L^2(\Gamma_1))\,,\quad g_1\,, g_2 \in L^2(0,T;\L2)\,.
\eeq
Then, Theorem \ref{theorem} ensures the existence of
\beq
  \label{solutions}
  u_1\,, u_2 \in L^2(0,T;\H1)\,, \quad v_1\,, v_2 \in H^1(0,T;\H1')\cap L^2(0,T;\H1)
\eeq
such that \eqref{thm1}--\eqref{thm3} hold for 
$(u_1\,, v_1\,,\xi_1)$ and $(u_2\,, v_2\,,\xi_2)$.
Taking the difference in \eqref{thm1} and testing by $z=u_1(t)-u_2(t)$, we obtain
\[
  \begin{split}
      &\left<\frac{\partial(v_1-v_2)}{\partial t}(t),(u_1-u_2)(t)\right>+\int_\Omega{\left|\nabla(u_1-u_2)(t)\right|^2\, dx}+
      \int_{\Gamma_1}{(\xi_1-\xi_2)(t)(u_1-u_2)(t)\,ds}\\
      &=\int_\Omega{(g_1-g_2)(t)(u_1-u_2)(t)\,dx}+
      \int_{\Gamma_1}{(h_1-h_2)(t)(u_1-u_2)(t)\,ds}\,;
  \end{split}
\]
hence, recalling \eqref{gamma_lin} and taking into account the monotonicity of $\beta$, this equation implies
\beq
  \label{est1}
  \begin{split}
    &\frac{\alpha}{2}\frac{d}{dt}\int_\Omega{\left|u_1(t)-u_2(t)\right|^2\, dx}+\nL2{\nabla(u_1(t)-u_2(t))}^2\\
    &\leq\int_\Omega{(g_1-g_2)(t)(u_1-u_2)(t)\,dx}+ \int_{\Gamma_1}{(h_1-h_2)(t)(u_1-u_2)(t)\,ds}\,.
  \end{split}
\eeq
If we now integrate \eqref{est1} with respect to time, thanks to \eqref{stimh1} and the Young inequality, for all $t\in[0,T]$ we have
\[
  \begin{split}
    &\frac{\alpha}{2}\nL2{u_1(t)-u_2(t)}^2+\l|{\nabla(u_1-u_2)}\r|^2_{L^2(0,t;\L2)}\leq
    \frac{\alpha}{2}\nL2{u_1^0-u_2^0}^2+\\
    &+\int_0^t{\left(\nL2{g_1(s)-g_2(s)}+C\nL2{h_1(s)-h_2(s)}\right)\nH1{u_1(s)-u_2(s)}\,ds}\\
    &\leq\left[\frac{\alpha}{2}\nL2{u_1^0-u_2^0}^2+
    \l|g_1-g_2\r|_{L^2(0,T;\L2)}^2+C^2\l|h_1-h_2|\r|_{L^2(0,T;L^2(\Gamma_1))}^2\right]+\\
    &+\frac{1}{2}\int_0^t{\l|u_1(s)-u_2(s)\r|_{\H1}^2\,ds}\,,
  \end{split}
\]
which implies
\beq
  \label{est2}
    \begin{split}
      \frac{\alpha}{2}\nL2{u_1(t)-u_2(t)}^2&+\frac{1}{2}\l|{\nabla(u_1-u_2)}\r|^2_{L^2(0,t;\L2)}\\
     &\leq\left[\ldots\right]+\frac{1}{2}\int_0^t{\nL2{u_1(s)-u_2(s)}^2\,ds}\,.
    \end{split}
\eeq
In particular, Gronwall Lemma ensures that
\beq
\begin{split}
  \nL2{u_1(t)-u_2(t)}^2&+\frac{1}{\alpha}\l|{\nabla(u_1-u_2)}\r|^2_{L^2(0,t;\L2)}\\
  &\leq\frac{2}{\alpha}\left[\ldots\right]e^{t/\alpha}\leq\frac{2}{\alpha}\left[\ldots\right]e^{T/\alpha} \quad\text{for all $t\in(0,T)$}\,.
\end{split}
\eeq
Hence, we have shown that
\beq
  \label{est3}
  \begin{split}
    &\l|u_1-u_2\r|_{L^\infty(0,T;\L2)\cap L^2(0,T;\H1)}^2
    \leq\frac{2e^{T/\alpha}}{\alpha\min\{1,1/\alpha\}}\\
    &\times\left[\frac{\alpha}{2}\nL2{u_1^0-u_2^0}^2+
    \l|g_1-g_2\r|_{L^2(0,T;\L2)}^2+C^2\l|h_1-h_2|\r|_{L^2(0,T;L^2(\Gamma_1))}^2\right]
  \end{split}
\eeq
and the continuous dependence result is proved.
The uniqueness is an easy consequence when we consider $u_0^1=u_0^2$, $h_1=h_2$ and $g_1=g_2$.



\end{document}